\documentclass{ijmart}
\usepackage[T1]{fontenc}
\usepackage{lmodern,amsmath,amssymb,mathtools,enumitem,xcolor,etoolbox}
\makeatletter\input{t1lmr.fd}\makeatother
\DeclareFontShape{T1}{lmr}{m}{scit}{<->ssub*lmr/m/scsl}{}
\usepackage[colorlinks=true,linkcolor=blue!45!black,citecolor=blue!45!black,urlcolor=blue!45!black]{hyperref}
\usepackage{xurl}
\makeatletter
\patchcmd{\@setauthors}
 {\if\authors\@empty\relax\else{\normalfont\tiny\@@by}}
 {\ifx\authors\@empty\else{\normalfont\tiny\@@by}\fi}
 {}{\ClassWarning{ijmart}{Author-block compatibility patch not applied}}
\makeatother
\newtheorem{theorem}{Theorem}[section]
\newtheorem{proposition}[theorem]{Proposition}
\newtheorem{lemma}[theorem]{Lemma}
\newtheorem{corollary}[theorem]{Corollary}
\newcommand{\N}{\mathbb N}
\newcommand{\Mod}{\operatorname{Mod}}
\newcommand{\Hom}{\operatorname{Hom}}
\newcommand{\Ext}{\operatorname{Ext}}
\newcommand{\Spec}{\operatorname{Spec}}
\newcommand{\QCoh}{\operatorname{QCoh}}
\newcommand{\Tow}{\operatorname{Tow}}
\newcommand{\ilim}{\varprojlim}
\newcommand{\Rlim}{\mathbf R\!\varprojlim}
\setlist[enumerate]{label=\textup{(\roman*)},leftmargin=2.5em}
\allowdisplaybreaks[2]
\newcommand{\A}{\mathcal A}
\newcommand{\End}{\operatorname{End}}
\newcommand{\pd}{\operatorname{pd}}
\newcommand{\dlim}{\varinjlim}
\newcommand{\im}{\operatorname{im}}
\newcommand{\coker}{\operatorname{coker}}
\newcommand{\Red}{\operatorname{Red}}
\newcommand{\Rprod}{\mathbf R\!\prod}
\title[Derived products and inverse limits]{Derived products and top-degree inverse limits}
\author[R. Zeng]{Rui Zeng}
\address{School of Mathematical Sciences, Sichuan Normal University,
Chengdu 610066, P.R. China}
\email{20250801006@stu.scinu.edu.cn}
\author[G. Dai]{Guocheng Dai}
\address{School of Mathematical Sciences, Sichuan Normal University,
Chengdu 610066, P.R. China}
\email{ggcdai@163.com}
\hypersetup{pdftitle={Derived products and top-degree inverse limits},
 pdfauthor={Rui Zeng and Guocheng Dai},
 pdfsubject={Top derived inverse limits, reduced products, and exact coresolutions}}
\fancypagestyle{firstpage}{\fancyhf{}\cfoot{\thepage}}
\begin{document}
\begin{abstract}
We compute top derived products and inverse limits on products of
affine lines with doubled origin over $B[s]$, for any nonzero
commutative ring $B$. For every $B$-module $M$, the two top obstruction
modules fit into an explicit exact sequence and contain
$M^{\N}/M^{(\N)}$ as a natural $B$-module direct summand.
This gives exact product dimensions, sharp vanishing bounds for
inverse sequences, and faithfulness in the coefficient module.
For a Grothendieck category of finite product dimension, we also obtain
degree shifts under adjoining commuting locally nilpotent endomorphisms
and prove that this construction preserves top-degree vanishing for
epimorphic inverse sequences. A fixed Gabriel--Popescu presentation
then describes the general top obstruction through actual Ext sequences
and separates two module-theoretic terms.
\end{abstract}
\maketitle

\section{Introduction}
Let $\A$ be a Grothendieck category. We write $d(\A)\leq\nu$
if the product functor $\Pi_I:\A^I\to\A$ has $R^q\Pi_I=0$
for every set $I$ and $q>\nu$.
This is the condition $\mathrm{AB4}^{*}\!-\nu$.
The least such $\nu\in\N$ is the cohomological dimension of products;
we set $d(\A)=\infty$ if no finite bound exists.
When $d(\A)\leq\nu<\infty$, derived limits of inverse sequences
vanish above degree $\nu+1$, and
\[
 R^{\nu+1}\ilim X_n\cong
 \coker\left(1-\sigma_X:R^\nu\prod_nX_n\longrightarrow
                         R^\nu\prod_nX_n\right),
\]
where $\sigma_X$ is induced by the bonding maps
(Proposition~\ref{prop:difference-triangle}).
When $d(\A)=\nu$, the top-degree vanishing question asks whether
this cokernel is zero for every epimorphic inverse sequence. Equivalently, one may ask for
all Mittag--Leffler inverse sequences
(Corollary~\ref{cor:ML-reduction}).
An affine cover may give an upper bound larger than $d(\A)$;
vanishing at that bound alone does not settle the question.

We compute these functors in a geometric family and determine how
the vanishing question changes under a nilpotent extension.
Let $B\ne0$ be a commutative ring and $r\geq1$. Let $X_r$ be the $r$-fold
product, over $\Spec B[s]$, of the affine line with doubled origin.
Our main result, Theorem~\ref{main:geometric}, computes
\[
 R^r\prod_{n\geq1}^{\QCoh(X_r)}\mathcal F_M,
 \qquad R^{r+1}\ilim^{\QCoh(X_r)}(\mathcal F_M,s),
 \qquad \mathcal F_M=\mathcal O_{X_r}\otimes_B M,
\]
for every $B$-module $M$.
The modules $W_r(M)$ and $V_r(M)$ defined in
\eqref{eq:W-module}--\eqref{eq:V-module} fit into
\[
 W_r(M)\xrightarrow{1-s\sigma}W_r(M)
 \xrightarrow{\Theta}V_r(M)\longrightarrow0,
 \qquad \Theta[(f_n)]=\left[\sum_{n\geq1}s^nf_n\right].
\]
Here $\sigma(f)_n=f_{n+1}$. Both modules contain
$M^{\N}/M^{(\N)}$ as a natural $B$-module
summand, with compatible embeddings. These splittings prove
nonvanishing and show that the two derived functors detect every
nonzero coefficient homomorphism.

The same calculation gives $d(\QCoh(X_r))=r$.
Derived limits of Mittag--Leffler inverse sequences vanish above $r$,
and a split epimorphic sequence attains degree $r$.
Arbitrary inverse sequences vanish above $r+1$, and multiplication
by $s$ on $\mathcal O_{X_r}$ attains degree $r+1$
(Corollary~\ref{cor:sharp-dimensions}).
Roos already constructed Mittag--Leffler sequences with a prescribed
single positive nonvanishing degree \cite{Roos}.
The finite \v Cech upper-bound method is due to Hogadi--Xu
\cite[Lemma~3.2 and Remark~3.3]{HX}; see also
\cite[Section~2]{PosRoos}.
For the family above, the reduced-product summands prove the matching
lower bounds. The calculation also identifies the two obstruction
modules for arbitrary coefficients and their comparison map.

Section~\ref{sec:nilpotent} gives an intrinsic counterpart.
For the category $\mathcal N_r(\A)$ of $r$ commuting locally
nilpotent endomorphisms, a relative cofree coresolution computes
derived limits by kernels and localization cokernels
(Theorem~\ref{thm:nilpotent-formula}).
If $\A\ne0$ has finite product dimension, then
$d(\mathcal N_r(\A))=d(\A)+r$.
Truncation diagrams shift derived limits by $r$, and epimorphic
top-degree vanishing holds in $\mathcal N_r(\A)$ exactly when it
holds in $\A$ (Theorem~\ref{thm:nilpotent-shift}).
The proof uses the classical cofree adjunction, an explicit contraction,
and a lower bound for products indexed by arbitrary sets.

Finally, Section~\ref{sec:gp} expresses the general top obstruction
through the Ext sequences of one fixed Gabriel--Popescu generator.
A short exact sequence separates the completion cokernel from the
first derived limit of the intersection of the finite-stage images.
These results leave the general top-degree vanishing question
unresolved.

\section{Derived products and inverse sequences}
\label{sec:preliminaries}
All index categories are small. Rings are unital, and modules are
left modules unless right modules are specified.
We use cochain complexes and injective resolutions in nonnegative
degrees, with $\N=\{0,1,\ldots\}$.
An inverse sequence is $X=(X_n,u_n)$, with $u_n:X_{n+1}\to X_n$.
We write $\Tow(\A)=\A^{\N^{\mathrm{op}}}$.
The sequence is epimorphic if every $u_n$ is an epimorphism, and
Mittag--Leffler if the images in each fixed $X_n$ eventually stabilize.
For modules we also write $\ilim^q=R^q\ilim$.
The bonding maps induce $\sigma_X:\prod_nX_n\to\prod_nX_n$,
with $n$th component $u_n\operatorname{pr}_{n+1}$.
We use the same notation for the induced map on derived products.

\begin{lemma}\label{lem:tower-injectives}
Evaluation preserves injective inverse sequences.
Every injective sequence $I$ is a direct summand of
$\prod_{n\geq0}r_nI_n$, where $r_nE$ is zero below $n$ and
constant with value $E$ from $n$ onwards.
In particular, the terms $I_n$ are injective and the bonding maps split.
If $F:\A\to\mathcal B$ is an exact functor between Grothendieck
categories and preserves injectives, its termwise extension
$\Tow(\A)\to\Tow(\mathcal B)$ preserves injectives.
\end{lemma}
\begin{proof}
Evaluation at $n$ has an exact left adjoint supported on $0,\ldots,n$
and an exact right adjoint $r_n$. Thus $I_n$ and $r_nI_n$ are injective.
The adjunction maps give a monomorphism
$I\to\prod_nr_nI_n$: its $n$th coordinate contains $1_{I_n}$.
This monomorphism splits because $I$ is injective.
At index $j$ the target is $\prod_{n\leq j}I_n$, with projection
bonding maps. The functor $F$ preserves these finite products and
the splitting, and sends the target to $\prod_nr_nF(I_n)$,
which is injective.
\end{proof}

\begin{proposition}\label{prop:difference-triangle}
Every inverse sequence $X$ has a natural distinguished triangle
\begin{equation}\label{eq:difference-triangle}
 \Rlim X\longrightarrow\Rprod_nX_n
 \xrightarrow{1-\sigma_X}\Rprod_nX_n\longrightarrow(\Rlim X)[1]
\end{equation}
in $D^+(\A)$.
If $d(\A)\leq\nu<\infty$, then
\begin{equation}\label{eq:top-cokernel}
 R^{\nu+1}\ilim X\cong
 \coker(1-\sigma_X:R^\nu\prod_nX_n\to R^\nu\prod_nX_n),
\end{equation}
and $R^q\ilim X=0$ for $q\geq\nu+2$.
\end{proposition}
\begin{proof}
Resolve $X$ by injectives $I^\bullet$ in $\Tow(\A)$.
Fix a degree $a$ and write $p_n:\prod_jI_j^a\to I_n^a$ for the
projections. Choose sections $v_n:I_n^a\to I_{n+1}^a$ of the bonding
maps $u_n$, as allowed by Lemma~\ref{lem:tower-injectives}.
Define morphisms $b_n:\prod_jI_j^a\to I_n^a$ recursively by
\[
 b_0=0,\qquad b_{n+1}=v_n(b_n-p_n).
\]
Then $b_n-u_nb_{n+1}=p_n$, so $(b_n)_n$ is a section of
$1-\sigma_{I^a}$. Each coordinate involves only finitely many
morphisms. Consequently
\[
 0\to\ilim_nI_n^\bullet\to\prod_nI_n^\bullet
 \xrightarrow{1-\sigma_I}\prod_nI_n^\bullet\to0
\]
is exact. Termwise injectivity identifies the three complexes with
the derived functors in \eqref{eq:difference-triangle}.
Its cohomology sequence gives the remaining assertions.
\end{proof}

\begin{lemma}\label{lem:partial-products}
For a family $(D_n)_{n\geq1}$, let $P_m=\prod_{1\leq n\leq m}D_n$
have projection bonding maps. Then
$\Rlim_mP_m\simeq\Rprod_nD_n$ naturally.
\end{lemma}
\begin{proof}
The partial-product functor is exact and sends an injective family
$(J_n)$ to the injective sequence $\prod_{n\geq1}r_nJ_n$.
Derive the identity $\ilim_mP_m=\prod_nD_n$.
\end{proof}

\begin{proposition}\label{prop:finite-coresolution}
Suppose $\A$ has a functorial exact coresolution
\[
 0\to\operatorname{Id}_{\A}\to C^0\to\cdots\to C^a\to0,
\]
where each $C^p$ is a finite sum of $F_\alpha U_\alpha$, and
$U_\alpha:\A\to\Mod R_\alpha$ has exact right adjoint $F_\alpha$,
with $U_\alpha$ also exact. Then $d(\A)\leq a$.
For every Mittag--Leffler inverse sequence $X$,
\begin{equation}\label{eq:ML-coresolution}
 \Rlim X\simeq
 [\ilim C^0X\longrightarrow\cdots\longrightarrow\ilim C^aX],
\end{equation}
so $R^q\ilim X=0$ for $q>a$.
Arbitrary inverse sequences vanish above $a+1$.
For an epimorphic inverse system indexed by a directed set $I$
with $\operatorname{cof}(I)\leq\aleph_n$, $n<\omega$, one has
$R^q\ilim_I X=0$ for $q>a+n$.
\end{proposition}
\begin{proof}
The exact right adjoints $F_\alpha$ preserve limits and injectives,
also after passage to diagram categories. Taking componentwise
injective resolutions therefore gives
\[
 R^q\Pi^{\A}(F_\alpha M)\cong F_\alpha R^q\Pi^{R_\alpha}M,
 \qquad R^q\ilim^{\A}(F_\alpha M)\cong
                         F_\alpha R^q\ilim^{R_\alpha}M.
\]
Thus, for every family $(X_i)$, the family $(C^pX_i)_i$ is
product-acyclic. The coresolution gives $d(\A)\leq a$.
Exactness of $U_\alpha$ preserves images and the Mittag--Leffler
condition. Thus every $C^pX$ is acyclic for inverse limit, by the
module Mittag--Leffler theorem, and the coresolution computes
$\Rlim X$ as in \eqref{eq:ML-coresolution}.
For arbitrary inverse sequences, the bounded coresolution gives
the spectral sequence
\[
 E_1^{p,q}=R^q\ilim C^pX\ \Longrightarrow\ R^{p+q}\ilim X
\]
has $0\leq p\leq a$ and $q=0,1$.
For the final assertion, the surjective Goblot theorem bounds the
module rows by $q\leq n$ \cite[Theorem~2.4]{VV}.
\end{proof}

\section{Top derived functors with arbitrary coefficients}
\label{sec:geometry}
Fix a nonzero commutative ring $B$ and $r\geq1$, and put $C=B[s]$.
Let $D_C$ be obtained by gluing two copies of $\Spec C[t]$ along
$\Spec C[t,t^{-1}]$. Set
\[
 X_r=D_C\times_{\Spec C}\cdots\times_{\Spec C}D_C
 \quad(r\text{ factors}),\qquad \A_r=\QCoh(X_r).
\]
For $S\subseteq[r]=\{1,\ldots,r\}$ and a $B$-module $M$, write
\[
\begin{gathered}
 R_S=B[s,t_1,\ldots,t_r,t_i^{-1}:i\in S],\qquad
 L_S(M)=M\otimes_B R_S,\\
 \widehat L_S(M)=M[t_1,\ldots,t_r,t_i^{-1}:i\in S][[s]],
 \qquad \mathcal F_M=\mathcal O_{X_r}\otimes_B M.
\end{gathered}
\]
The symbol $[[s]]$ denotes all formal series in the displayed
coefficient module. Let $j:T=\Spec R_{[r]}\hookrightarrow X_r$
be the common torus inclusion, and define
\begin{align}
 W_r(M)&=\frac{\prod_{n\geq1}L_{[r]}(M)}
 {\displaystyle\sum_{h=1}^r
  (\prod_{n\geq1}L_{[r]\setminus\{h\}}(M))[t_h^{-1}]},
 \label{eq:W-module}\\
 V_r(M)&=\frac{\widehat L_{[r]}(M)}
 {\displaystyle L_{[r]}(M)+\sum_{h=1}^r
  \widehat L_{[r]\setminus\{h\}}(M)[t_h^{-1}]}.
 \label{eq:V-module}
\end{align}
The denominator terms are identified with their natural images in
the numerators. They are $R_{[r]}$-submodules: in the $h$th term,
$t_j^{\pm1}$ already act for $j\ne h$, and $t_h$ is inverted by the
displayed localization. In \eqref{eq:V-module}, completion precedes
localization.

Write $\Red(M)=\prod_{n\geq1}M/\bigoplus_{n\geq1}M$.

\begin{theorem}\label{main:geometric}
For every $B$-module $M$, there are natural isomorphisms
\begin{align}
 R^r\prod_{n\geq1}^{\A_r}\mathcal F_M
 &\cong j_*\widetilde{W_r(M)},\label{eq:W-derived}\\
 R^{r+1}\ilim^{\A_r}(\mathcal F_M,s)
 &\cong j_*\widetilde{V_r(M)}.\label{eq:V-derived}
\end{align}
The shift $\sigma(f)_n=f_{n+1}$ induces a natural exact sequence
of $R_{[r]}$-modules
\begin{equation}\label{eq:theta-sequence}
 W_r(M)\xrightarrow{1-s\sigma}W_r(M)
 \xrightarrow{\Theta}V_r(M)\longrightarrow0,
 \qquad \Theta[(f_n)]=\left[\sum_{n\geq1}s^nf_n\right].
\end{equation}
There are natural $B$-linear split embeddings
$\iota_W:\Red(M)\to W_r(M)$ and $\iota_V:\Red(M)\to V_r(M)$,
given by
\begin{align*}
 \iota_W[(a_n)]&=[(a_n(t_1\cdots t_r)^{-n})_n],\\
 \iota_V[(a_n)]&=\left[\sum_{n\geq1}s^na_n(t_1\cdots t_r)^{-n}\right].
\end{align*}
They satisfy $\Theta\iota_W=\iota_V$.
In particular, both derived functors are nonzero for every $M\ne0$.
\end{theorem}
\begin{proof}
\emph{The cubical coresolution.}
The standard affine charts $U_\varepsilon$,
$\varepsilon\in\{+,-\}^r$, have ring $R_\varnothing$.
Their intersections are affine localizations. Hence $X_r$ is
quasi-compact and semi-separated, and $\A_r$ is Grothendieck
\cite[Tag~077P]{Stacks}.
For each coordinate $h$, pull back the two-chart cover of $D_C$.
Denote its opens by $V_{h,+},V_{h,-}$ and their intersection by
$V_{h,*}$. Each inclusion is affine: on a standard chart its
inverse image is either the whole chart or $D(t_h)$.
Thus $F_V=(j_V)_*j_V^*$ is exact, and open base change gives
$F_VF_W\cong F_{V\cap W}$.

For each $h$, the restriction difference gives an exact complex
\[
 0\to\operatorname{Id}\to F_{V_{h,+}}\oplus F_{V_{h,-}}
 \to F_{V_{h,*}}\to0.
\]
Compose these $r$ augmented complexes and totalize.
The component replacing $\alpha_h=+$ by $*$ has sign $(-1)^{a_h}$;
that replacing $\alpha_h=-$ by $*$ has sign $(-1)^{a_h+1}$, where
$a_h=|\{k<h:\alpha_k=*\}|$.
Every square anticommutes, so the total differential squares to zero.
To prove exactness, take cohomology in one coordinate at a time.
The two-term complex for that coordinate has only degree-zero
cohomology, equal to the identity functor. All remaining functors
are exact, so repeating this argument leaves the identity in degree
zero and no higher cohomology. We obtain an exact augmented
coresolution with terms
\begin{equation}\label{eq:cubical-coresolution}
 C^p(F)=\bigoplus_{\substack{\alpha\in\{+,-,*\}^r\\
                         |\{h:\alpha_h=*\}|=p}}
 (j_\alpha)_*(F|_{V_\alpha}),\qquad 0\leq p\leq r,
\end{equation}
where $V_\alpha=\Spec R_{\{h:\alpha_h=*\}}$.
Restriction to each affine $V_\alpha$ is exact and has the exact
right adjoint $(j_\alpha)_*$. Hence this coresolution satisfies
Proposition~\ref{prop:finite-coresolution}.

\emph{The two top cohomology objects.}
Apply products to \eqref{eq:cubical-coresolution} for the constant
family $\mathcal F_M$. Its terms are product-acyclic.
Since $(j_\alpha)_*$ preserves products, the product of its terms is
\[
 \prod_n^{\A_r}(j_\alpha)_*\widetilde{L_{S_\alpha}(M)}
 \cong (j_\alpha)_*\widetilde{\prod_n L_{S_\alpha}(M)},
 \qquad S_\alpha=\{h:\alpha_h=*\}.
\]
We now restrict these products to $U_\varepsilon$.
The top term is $\prod_nL_{[r]}(M)$.
For a missing coordinate $h$, the two degree-$(r-1)$ terms
contribute respectively
\[
 \prod_nL_{[r]\setminus\{h\}}(M),\qquad
 (\prod_nL_{[r]\setminus\{h\}}(M))[t_h^{-1}].
\]
The first comes from the chart label agreeing with $\varepsilon_h$;
the second comes from the opposite label. Their sum gives exactly
the denominator in \eqref{eq:W-module}.
Thus the top cohomology on each standard chart is $W_r(M)$.

On a pairwise chart intersection, invert the $t_h$ for which
the chart labels differ. These already act invertibly on $W_r(M)$.
The identifications all come from the same global top term;
affine base change for $j$ therefore glues them to
$j_*\widetilde{W_r(M)}$, proving \eqref{eq:W-derived}.

Next, the module sequence $(L_S(M),s)$ has zero limit and
\begin{equation}\label{eq:affine-completion}
 \ilim{}^1(L_S(M),s)\cong\widehat L_S(M)/L_S(M).
\end{equation}
Indeed, multiplication by $s$ is injective and
$\bigcap_ns^nL_S(M)=0$.
The map $(a_n)_{n\geq0}\mapsto\sum_{n\geq0}s^na_n$ modulo
$L_S(M)$ is surjective: for a series $\sum_ns^nc_n$, choose $a_n=c_n$
in the polynomial coefficient module.
Telescoping sends $(b_n-sb_{n+1})_n$ to $b_0$ modulo $L_S(M)$.
If the series equals $b_0\in L_S(M)$, then
\[
 b_n=s^{-n}\left(b_0-\sum_{k<n}s^ka_k\right)\in L_S(M)
\]
provides a preimage under the difference map.
The numerator has its first $n$ coefficients zero; thus $s^{-n}$
here means division by a power of $s$ inside the polynomial module,
not a localization. This proves \eqref{eq:affine-completion} for arbitrary $M$.

Apply $\Rlim$ to \eqref{eq:cubical-coresolution} for
$(\mathcal F_M,s)$. By exact adjunction and
\eqref{eq:affine-completion}, the spectral sequence has only row
$q=1$. Its top cohomology is the degree-$r$ cokernel of the complex
with terms $(j_\alpha)_*\widetilde{\widehat L_{S_\alpha}(M)/L_{S_\alpha}(M)}$,
where $S_\alpha=\{h:\alpha_h=*\}$.
On a standard chart, the images missing $h$ come from
$\widehat L_{[r]\setminus\{h\}}(M)/L_{[r]\setminus\{h\}}(M)$
and its $t_h$-localization. Their sum gives \eqref{eq:V-module}.
As above, all differing chart coordinates act invertibly on the
quotient, and the local identifications glue to
$j_*\widetilde{V_r(M)}$. This proves \eqref{eq:V-derived}.

\emph{The comparison sequence.}
The shift preserves every denominator of $W_r(M)$.
The series defining $\Theta$ exists coefficientwise, since only
finitely many $f_n$ contribute to each $s$-degree.
A family in the $h$th denominator has a uniform lower bound on
its $t_h$-exponents, so its series lies in the $h$th denominator
of $V_r(M)$. Hence $\Theta$ is well defined and $R_{[r]}$-linear.
For surjectivity, write a series as $\sum_{n\geq0}s^nc_n$ with
$c_n\in M[t_1^{\pm1},\ldots,t_r^{\pm1}]$.
Its constant term belongs to $L_{[r]}(M)$, so choosing $f_n=c_n$
for $n\geq1$ gives its class under $\Theta$.
Telescoping gives $\Theta((1-s\sigma)b)=[sb_1]=0$.

Conversely, suppose
$\sum_{n\geq1}s^nf_n=p+\sum_hg_h$, where $p\in L_{[r]}(M)$
and $g_h\in\widehat L_{[r]\setminus\{h\}}(M)[t_h^{-1}]$.
Put
\[
 f'_n=f_n-\sum_{h=1}^r[s^n]g_h\qquad(n\geq1).
\]
For each $h$, the family $([s^n]g_h)_n$ has a uniform lower bound
on its $t_h$-exponents, so it lies in the $h$th denominator of
$W_r(M)$. Hence $f'$ represents the same class as $f$, and
\[
 P=\sum_{n\geq1}s^nf'_n=p+\sum_h[s^0]g_h\in L_{[r]}(M).
\]
Set
\[
 b_n=s^{-n}\left(P-\sum_{1\leq k<n}s^kf'_k\right).
\]
The numerator has no terms below degree $n$, so $b_n\in L_{[r]}(M)$,
and $b_n-sb_{n+1}=f'_n$. This proves \eqref{eq:theta-sequence}.

\emph{The reduced-product summands.}
A finitely supported sequence maps to zero under $\iota_W$ because
its negative exponents have a common bound, and under $\iota_V$
because its series is a polynomial.
Define $B$-linear retractions by
\[
\begin{split}
 \rho_W[(f_n)]&=
  \bigl[([s^0t_1^{-n}\cdots t_r^{-n}]f_n)_n\bigr],\\
 \rho_V[F]&=
  \bigl[([s^nt_1^{-n}\cdots t_r^{-n}]F)_n\bigr].
\end{split}
\]
The inner brackets mean coefficient extraction in $M$.
Each denominator maps to a finitely supported sequence: use its
uniform exponent bound, and the finite $s$-degree for the polynomial
part of $V_r(M)$.
Thus the maps are well defined and natural in $M$, and
\[
 \rho_W\iota_W=\rho_V\iota_V=1,\qquad \Theta\iota_W=\iota_V.
\]
A nonzero constant sequence survives in $\Red(M)$, proving nonvanishing.
\end{proof}

The splittings are $B$-linear, and their compatibility with $\Theta$
concerns the embeddings. The retractions need not commute with
$\Theta$: for $0\ne x\in M$ and
$f_n=sx(t_1\cdots t_r)^{-(n+1)}$, one has
$\rho_W[f]=0$ but $\rho_V\Theta[f]=[(x)_n]\ne0$.

\begin{corollary}\label{cor:sharp-dimensions}
One has $d(\A_r)=r$.
Every Mittag--Leffler inverse sequence in $\A_r$ has zero derived
limits above $r$, and a split epimorphic sequence attains degree $r$.
Every inverse sequence has zero derived limits above $r+1$, and
$(\mathcal O_{X_r},s)$ attains degree $r+1$.
For an epimorphic inverse system of cofinality at most $\aleph_n$,
$n<\omega$, one has $R^q\ilim^{\A_r}=0$ for $q>r+n$.
\end{corollary}
\begin{proof}
The cubical coresolution and Proposition~\ref{prop:finite-coresolution}
give all upper bounds. Take $M=B$ in Theorem~\ref{main:geometric}.
Since $\Red(B)\ne0$, the two top objects are nonzero, proving
$d(\A_r)=r$ and attainment in degree $r+1$.
Lemma~\ref{lem:partial-products} identifies the degree-$r$ object
with the derived limit of
$P_m=\mathcal O_{X_r}^{\oplus m}$ with projection bonding maps.
This sequence is split epimorphic.
The images $s^m\mathcal O_{X_r}$ strictly decrease on every standard
chart, so the sequence attaining degree $r+1$ is not Mittag--Leffler.
\end{proof}

For $M=B$, the two nonzero classes are represented by
$((t_1\cdots t_r)^{-n})_n$ and
$\sum_{n\geq1}s^n(t_1\cdots t_r)^{-n}$: they are the images of
the constant sequence $1_B$ under $\iota_W$ and $\iota_V$.

\begin{corollary}\label{cor:faithful-coefficients}
Both functors in \eqref{eq:W-derived}--\eqref{eq:V-derived}, viewed
as functors from $\Mod B$ to $\A_r$, are faithful.
\end{corollary}
\begin{proof}
If $a:M\to N$ is nonzero, choose $x$ with $a(x)\ne0$.
Then $\Red(a)$ is nonzero on the class of $(x)_n$.
Naturality gives
$\rho_{W,N}W_r(a)\iota_{W,M}=\Red(a)$, and similarly for $V_r$.
Thus both module maps are nonzero. Apply the faithful functor $j_*$,
using $j^*j_*=1$.
\end{proof}

\section{Locally nilpotent endomorphisms}
\label{sec:nilpotent}
Let $\mathcal N(\A)$ consist of pairs $(M,u)$ with
$M=\bigcup_{m\geq1}\ker u^m$, and morphisms commuting with $u$.
Put $\mathcal N_0(\A)=\A$ and
$\mathcal N_r(\A)=\mathcal N(\mathcal N_{r-1}(\A))$.
Thus its objects carry $r$ commuting locally nilpotent endomorphisms.
Let $U:\mathcal N(\A)\to\A$ be the forgetful functor, and put
\[
 E(D)=\bigoplus_{b\geq1}D e_b,\qquad
 te_1=0,\quad te_b=e_{b-1}\quad(b>1).
\]
Write $E_r$ for the iterated cofree functor.
For $m\geq1$, let $K_m^{(r)}(D)$ be the direct sum of copies
$D t^{\mathbf a}$ indexed by $0\leq a_j<m$.
The operator $t_j$ is the identity from the summand $D t^{\mathbf a}$
to $D t^{\mathbf a+\mathbf e_j}$ when $a_j<m-1$, and zero otherwise.
The maps $K_{m+1}^{(r)}\to K_m^{(r)}$ are truncations.

The adjunction below is classical; compare \cite[Section~2]{PosCoalg}.
We give the relative coresolution explicitly. Related graded comodule
calculations appear in \cite{Salch}.

\begin{lemma}\label{lem:cofree-coresolution}
The category $\mathcal N(\A)$ is Grothendieck.
The functor $U$ is exact, and $E$ is its exact faithful right adjoint.
Every $(M,u)$ has a natural short exact sequence
\begin{equation}\label{eq:cofree-coresolution}
 0\to(M,u)\xrightarrow{\eta}E(UM)
 \xrightarrow{t-E(u)}E(UM)\to0,
 \qquad \eta_b=u^{b-1}.
\end{equation}
For every family $(D_j)$,
\begin{equation}\label{eq:cofree-products}
 R^q\prod_j^{\mathcal N(\A)}E(D_j)
 \cong E(R^q\prod_j^\A D_j).
\end{equation}
The same identity holds for derived limits over any small category.
\end{lemma}
\begin{proof}
Kernels, cokernels and filtered colimits are computed in $\A$;
exactness of filtered colimits shows that the induced operators
remain locally nilpotent. Thus $\mathcal N(\A)$ is abelian and
satisfies AB5.
If $G$ generates $\A$, then $\bigoplus_{m\geq1}K_m^{(1)}(G)$
generates $\mathcal N(\A)$. Indeed, a map $a:G\to\ker u^m$
extends to an equivariant map $K_m^{(1)}(G)\to(M,u)$ by using
$u^ka$ on the summand $G t^k$. The images of such maps contain all
$\ker u^m$, since $G$ is a generator, and these kernels exhaust $M$.
Hence $\mathcal N(\A)$ is Grothendieck and $U$ is exact.

The morphism corresponding to $f:UM\to D$ under the adjunction
has $b$th component $fu^{b-1}$. Define it on each $\ker u^m$
and then on their union.
The functor $E$ is exact by AB4 and faithful by its first summand.
It preserves injectives and limits, also in diagram categories;
this proves \eqref{eq:cofree-products} and its limit analogue.

On $\bigoplus_{b=1}^N(UM)e_b$, the kernel in
\eqref{eq:cofree-coresolution} is the image of $\ker u^N$ under
$x\mapsto(x,ux,\ldots,u^{N-1}x)$. The first-coordinate projection
is its inverse. Taking the filtered union identifies the full kernel
with $(M,u)$ via $\eta$.
For surjectivity, let $h(e_b)=e_{b+1}$ on the underlying coproduct.
Then $(t-E(u))h=1-E(u)h$, so
\[
 h\sum_{k\geq0}(E(u)h)^k
\]
is a right inverse in $\A$.
Since $h$ commutes with $E(u)$, the sum is finite on each
$\ker E(u)^m$, and these subobjects exhaust $E(UM)$.
\end{proof}

For an endomorphism $v:D\to D$, set
\[
 \begin{gathered}
 D[v^{-1}]=\dlim(D\xrightarrow vD\xrightarrow v\cdots),\qquad
 \Gamma_vD=\bigcup_{m\geq1}\ker v^m,\\
 \Phi_vD=\coker(D\to D[v^{-1}]).
 \end{gathered}
\]
The induced operators make $\Gamma_vD$ and $\Phi_vD$ objects
of $\mathcal N(\A)$, even when $v$ is not locally nilpotent on $D$.

\begin{theorem}\label{thm:nilpotent-formula}
Let $I$ be a small category and $X:I^{\mathrm{op}}\to\mathcal N(\A)$
a diagram. Put $D_q=R^q\ilim_I^\A UX$, and let $v_q$ be the induced
endomorphism. There are natural short exact sequences in $\mathcal N(\A)$
\begin{equation}\label{eq:nilpotent-limits}
 0\to\Phi_{v_{q-1}}D_{q-1}
 \to R^q\ilim_I^{\mathcal N(\A)}X
 \to\Gamma_{v_q}D_q\to0\qquad(q\geq0),
\end{equation}
where $D_{-1}=0$.
If one exponent annihilates the operators on every $X_i$, then
$UR^q\ilim_I^{\mathcal N(\A)}X\cong D_q$ for all $q$.
\end{theorem}
\begin{proof}
Put $E_{\leq N}(D)=\bigoplus_{b=1}^ND e_b$.
For $f=t-E(v)$ on $E(D)$, one has
\[
 \ker f\cap E_{\leq N}(D)\cong\ker v^N,
 \qquad x\longmapsto(x,vx,\ldots,v^{N-1}x).
\]
These are identities between morphisms to finite direct sums:
the equations are $x_{b+1}=vx_b$ and $v x_N=0$, and the inverse
is first-coordinate projection. Since $E(D)$ is the filtered union
of the $E_{\leq N}(D)$, exactness of filtered colimits gives
$\ker f\cong\Gamma_vD$.

Let $i_n:D\to D[v^{-1}]$ be the telescope maps, so $i_n=i_{n+1}v$.
The morphism $E(D)\to\Phi_vD$ with $b$th summand map $[i_b]$
kills $f$, since $[i_{b-1}]-[i_b]v=0$ and $[i_0]=0$.
Conversely, the summand maps $\bar e_b:D\to\coker f$ satisfy
$\bar e_1v=0$ and $\bar e_b=\bar e_{b+1}v$.
Assign $i_0\mapsto0$ and $i_b\mapsto\bar e_b$ for $b\geq1$.
These relations give inverse morphisms between $\coker f$ and
$\Phi_vD$, compatible with the operators.
Each image of $i_n(D)$ is killed by $v^n$ in the quotient,
proving local nilpotence. The identity
$v[i_b]=[i_{b-1}]$, with $[i_0]=0$, verifies compatibility with $t$.

Apply $\Rlim$ to \eqref{eq:cofree-coresolution} for $X$.
Lemma~\ref{lem:cofree-coresolution} identifies the induced maps
with $t-E(v_q)$ on $E(D_q)$.
The long exact sequence gives \eqref{eq:nilpotent-limits}.
A uniform exponent $m$ gives $v_q^m=0$, so
$D_q[v_q^{-1}]=0$ and $\Gamma_{v_q}D_q=D_q$.
\end{proof}

In particular, if $d(\A)\leq\nu<\infty$ and $X$ is an inverse sequence,
\begin{equation}\label{eq:nilpotent-top}
 R^{\nu+2}\ilim^{\mathcal N(\A)}X
 \cong\Phi_{v_{\nu+1}}(R^{\nu+1}\ilim^\A UX).
\end{equation}

\begin{theorem}\label{thm:nilpotent-shift}
Let $r\geq1$. For a small directed set $I$ and
$Y:I^{\mathrm{op}}\to\A$, put
$X_{i,m}=K_m^{(r)}(Y_i)$. Then
\begin{equation}\label{eq:nilpotent-shift}
 \Rlim_{I\times\N_{\geq1}}^{\mathcal N_r(\A)}X
 \simeq E_r(\Rlim_I^\A Y)[-r]
\end{equation}
naturally, and $X$ is epimorphic whenever $Y$ is.
If $\A\ne0$ and $d(\A)=\nu<\infty$, then
$d(\mathcal N_r(\A))=\nu+r$.
Moreover, the following are equivalent:
\begin{enumerate}
\item every epimorphic inverse sequence in $\A$ has
$R^{\nu+1}\ilim=0$;
\item every epimorphic inverse sequence in $\mathcal N_r(\A)$
has $R^{\nu+r+1}\ilim=0$.
\end{enumerate}
\end{theorem}
\begin{proof}
We first take $r=1$.
Let $\mathcal L(D)$ be the sequence $(E(D),t)$.
By \eqref{eq:cofree-products}, its derived difference map on
$E(R^q\prod_m^\A D)$ is $1-t\sigma_q$, where $\sigma_q$ deletes
the first product coordinate.
The operator $t$ is locally nilpotent and commutes with $\sigma_q$.
On $\ker t^N$, the inverse is $\sum_{k=0}^{N-1}(t\sigma_q)^k$.
These inverses agree on successive kernels, whose union is the
whole object, so $1-t\sigma_q$ is an isomorphism.
Proposition~\ref{prop:difference-triangle} gives $\Rlim\mathcal L(D)=0$.
The exact sequence
\[
 0\to(K_m^{(1)}(D))_m\to\mathcal L(D)
 \xrightarrow{(t^m)_m}\operatorname{const}E(D)\to0
\]
embeds $D t^a$ as $D e_{m-a}$ at stage $m$.
On sequences indexed by $\N_{\geq1}$, the constant-sequence
functor is $r_1$, so it is exact and preserves injectives. Thus
$\Rlim_mK_m^{(1)}(D)\simeq E(D)[-1]$.
The inner limit preserves injectives as right adjoint to an exact
constant-diagram functor. Derive iterated limits and use the limit
analogue of \eqref{eq:cofree-products} to obtain
\eqref{eq:nilpotent-shift} for $r=1$.

For product dimension, \eqref{eq:cofree-coresolution} gives the upper
bound $d(\mathcal N(\A))\leq\nu+1$.
Choose a family $(D_j)$ with $Z=R^\nu\prod_jD_j\ne0$.
Such a family exists by minimality of $\nu$; when $\nu=0$, take a
single nonzero object. Put
$W=R^\nu\prod_{j,m}D_j$.
Repetition and projection induce $\Delta:Z\to W$ and
$\pi_m:W\to Z$ with $\pi_m\Delta=1$.
Apply derived products to
$0\to K_m^{(1)}(D_j)\to E(D_j)\xrightarrow{t^m}E(D_j)\to0$,
indexed by $(j,m)$. The relevant part of the long exact sequence is
\[
 E(W)\xrightarrow fE(W)\longrightarrow
 R^{\nu+1}\prod_{j,m}^{\mathcal N(\A)}K_m^{(1)}(D_j)
 \longrightarrow E\!\left(R^{\nu+1}\prod_{j,m}^{\A}D_j\right)=0.
\]
Here we used \eqref{eq:cofree-products} and $d(\A)=\nu$.
Consequently
\[
 R^{\nu+1}\prod_{j,m}^{\mathcal N(\A)}K_m^{(1)}(D_j)
 \cong\coker(f:E(W)\to E(W)),\qquad
 E(\pi_m)f=t^mE(\pi_m).
\]
This cokernel contains $Z$ through $e_1\Delta$.
Indeed, for $E_{\leq N}(W)=\bigoplus_{b=1}^NW e_b$ and $m\geq N$,
$E(\pi_m)$ kills $f(E_{\leq N}(W))$ but is monic on
$e_1\Delta(Z)$.
Thus
$f(E_{\leq N}(W))\cap e_1\Delta(Z)=0$ for each $N$.
The images $f(E_{\leq N}(W))$ form an increasing union equal to
$\im f$. AB5 commutes this union with intersection by the fixed
subobject $e_1\Delta(Z)$, giving
\[
 \im f\cap e_1\Delta(Z)
 =\bigcup_N\bigl(f(E_{\leq N}(W))\cap e_1\Delta(Z)\bigr)=0.
\]
Therefore $Z$ embeds in $\coker f$, proving the lower bound.

Assume (i) and apply \eqref{eq:cofree-coresolution} to an epimorphic
sequence $M$ in $\mathcal N(\A)$.
The sequence $UM$ is epimorphic; its cofree images have zero derived
limits in degrees $\nu+1$ and above. The long exact sequence gives (ii).
Conversely, for an epimorphic $Y$ in $\A$, the diagonal sequence
$(K_n^{(1)}(Y_n))_n$ is epimorphic and has degree-$\nu+2$ derived
limit $E(R^{\nu+1}\ilim Y)$ by \eqref{eq:nilpotent-shift}.
The diagonal is cofinal, and restriction preserves injectives because
its left Kan extension is evaluation at the maximum of the two indices.
This extension is exact, so the comparison holds for derived limits.
Faithfulness of $E$ gives (i).

For general $r$, first use independent indices $m_1,\ldots,m_r$.
At the $k$th step apply the one-operator result over
$\mathcal N_{k-1}(\A)$. The truncation functors in the other
variables are exact and commute with the cofree functor in this
variable. Iteration gives
\[
 \Rlim_{I\times\N_{\geq1}^r}
 K_{m_r}^{(1)}\cdots K_{m_1}^{(1)}(Y_i)
 \simeq E_r(\Rlim_I^{\A}Y)[-r].
\]
The iterated truncation is the direct sum of $Y_i t^{\mathbf a}$
with $0\leq a_j<m_j$; on the diagonal it is $K_m^{(r)}(Y_i)$.
For the diagonal $\delta:\N_{\geq1}^{\mathrm{op}}\to
(\N_{\geq1}^{\mathrm{op}})^r$, left Kan extension sends a sequence
$Z$ to the diagram with value $Z_{\max(m_1,\ldots,m_r)}$.
It is exact, so restriction preserves injectives. The ordinary
cofinality isomorphism therefore also holds for derived limits,
and gives \eqref{eq:nilpotent-shift} with one truncation index.
At each step the product dimension increases by one and the
corresponding epimorphic vanishing conditions are equivalent.
This proves the remaining assertions.
\end{proof}

For a ring $A$, put $R=A[t_1,\ldots,t_r]$ with central variables.
Then $\mathcal N_r(\Mod A)$ is the category of
$(t_1,\ldots,t_r)$-power-torsion left $R$-modules.
For finitely many commuting operators, individual local nilpotence
is equivalent to every element being annihilated by some power of
this ideal.
In this description,
\[
\begin{gathered}
 K_m^{(r)}(D)=\bigl(R/(t_1^m,\ldots,t_r^m)\bigr)\otimes_A D,\\
 E_r(D)=\bigoplus_{b_1,\ldots,b_r\geq1}D t_1^{-b_1}\cdots t_r^{-b_r}.
\end{gathered}
\]

\begin{corollary}\label{cor:polynomial-torsion}
For every nonzero ring $A$ and $r\geq1$, products in
$\mathcal N_r(\Mod A)$ have cohomological dimension $r$.
Derived limits of epimorphic inverse sequences vanish above $r$,
and arbitrary inverse sequences vanish above $r+1$.
Both bounds are attained.
\end{corollary}
\begin{proof}
Products of $A$-modules are exact, and epimorphic module sequences
have zero first derived limit.
Theorem~\ref{thm:nilpotent-shift} gives the product dimension and
the epimorphic bound; Proposition~\ref{prop:difference-triangle}
gives the unrestricted bound.
The sequence $K_m^{(r)}(A)$ attains degree $r$, with value $E_r(A)$.
For the other bound use $Y=(A[s],s)$, with $s$ central.
The difference-map calculation in \eqref{eq:affine-completion},
which applies to arbitrary coefficient modules, gives $\ilim Y=0$ and $\ilim^1Y=A[[s]]/A[s]\ne0$.
Thus the diagonal sequence $K_n^{(r)}(A[s])$, with bonding map
$s$ times truncation, has nonzero derived limit
$E_r(A[[s]]/A[s])$ in degree $r+1$.
\end{proof}

\section{The Gabriel--Popescu obstruction}
\label{sec:gp}
Fix an exact localization
$Q:\Mod\text{-}R\rightleftarrows\A:S$, with $S$ fully faithful,
and put $G=Q(R_R)$.
Such a presentation exists by Gabriel--Popescu.
All Ext groups below carry the right $R$-action induced by
$R\to\End_\A(G)$.

\begin{proposition}\label{prop:GP-formulas}
For every family $(X_i)$ and $q\geq0$,
\begin{equation}\label{eq:GP-products}
 R^q\prod_i^\A X_i\cong
 Q\!\left(\prod_i\Ext_\A^q(G,X_i)\right).
\end{equation}
Every inverse sequence $X$ has a natural short exact sequence
\begin{equation}\label{eq:GP-Milnor}
 0\to Q\ilim_n^1\Ext_\A^{q-1}(G,X_n)
 \to R^q\ilim_nX_n\to Q\ilim_n\Ext_\A^q(G,X_n)\to0,
\end{equation}
where negative Ext groups are zero.
If $d(\A)\leq\nu<\infty$, then
\begin{equation}\label{eq:GP-top}
 R^{\nu+1}\ilim X\cong Q\ilim_n^1\Ext_\A^\nu(G,X_n).
\end{equation}
If $\pd_\A G\leq\nu$, every Mittag--Leffler inverse sequence
has zero derived limits in degrees $q>\nu$.
\end{proposition}
\begin{proof}
Adjunction identifies $S$ with $\Hom_\A(G,-)$, including its
right $R$-action. Since $S$ preserves products and $QS\cong1_\A$,
for every family $(J_i)$ one has
\[
 \prod_i^{\A}J_i\cong Q\!\left(\prod_i^{\Mod\text{-}R}SJ_i\right).
\]
Apply this identity to componentwise injective resolutions of $(X_i)$.
Module products and $Q$ are exact, so taking cohomology gives
\eqref{eq:GP-products}; compare \cite{Virili}.
This argument does not require $Q$ to preserve arbitrary products.

Now resolve $X$ by injectives $I^\bullet$ in $\Tow(\A)$.
The terms $I_n^a$ are injective by Lemma~\ref{lem:tower-injectives}.
The termwise functor $S$ preserves injective sequences because it
is right adjoint to the exact termwise functor $Q$.
Moreover, $S$ preserves limits and $QS\cong1_\A$, whence
\[
 \ilim^{\A} I^\bullet\cong Q(\ilim^{\Mod\text{-}R}S I^\bullet).
\]
The hypercohomology spectral sequence for the complex $S I^\bullet$ is
\[
 E_2^{p,q}=\ilim_n^p\Ext_\A^q(G,X_n)
 \ \Longrightarrow\ H^{p+q}(\ilim_n S I_n^\bullet).
\]
It has only columns $p=0,1$, since inverse limits of module sequences
have cohomological dimension at most one \cite[Chapter~2]{Jensen}.
Thus no differential survives from the second page. Applying the
exact functor $Q$ to its two-step filtration gives \eqref{eq:GP-Milnor}.

If $d(\A)\leq\nu$, then
$Q\prod_n\Ext_\A^{\nu+1}(G,X_n)=0$ by \eqref{eq:GP-products}.
The inverse limit is a submodule of this product, and the kernel
of an exact localization is closed under subobjects.
Thus \eqref{eq:GP-Milnor} gives \eqref{eq:GP-top}.

If $\pd G\leq\nu$, the functor $F=\Ext_\A^\nu(G,-)$ preserves
epimorphisms. For every transition,
\[
 \im(FX_m\to FX_n)=
 \im\bigl(F(\im(X_m\to X_n))\to FX_n\bigr).
\]
Hence $FX$ is Mittag--Leffler and $\ilim^1FX=0$.
Equation~\eqref{eq:GP-Milnor} then vanishes in degree $\nu+1$.
In larger degrees both Ext terms vanish, proving the assertion.
\end{proof}

\begin{corollary}\label{cor:ML-reduction}
Every Mittag--Leffler inverse sequence has the same derived limit
as its epimorphic sequence of stable images.
Consequently, top-degree vanishing for epimorphic inverse sequences
is equivalent to top-degree vanishing for Mittag--Leffler inverse sequences.
\end{corollary}
\begin{proof}
Call $P$ pro-zero if for every $n$ some transition $P_m\to P_n$ is zero.
Then each sequence $\Ext^q_\A(G,P_n)$ is pro-zero, so its module
limit and first derived limit vanish.
Equation~\eqref{eq:GP-Milnor} gives $\Rlim P=0$.
For a Mittag--Leffler sequence $X$, let $Y\subseteq X$ be its stable
images. Choosing a common stage beyond stabilization at two adjacent
indices shows that $Y$ is epimorphic. The quotient $X/Y$ is pro-zero.
The derived-limit triangle gives $\Rlim Y\simeq\Rlim X$.
\end{proof}

Thus the general top-degree question can be tested on the actual
Ext sequences of this one fixed generator.
Individual groups $\Ext_\A^j(G,X)$ are annihilated by $Q$ for $j>0$,
but their products or first derived limits need not be.
A completion separates two contributions to \eqref{eq:GP-top}.
For an inverse sequence $V$ of right $R$-modules, put
\[
 (I^mV)_n=\im(V_{n+m}\to V_n),\qquad
 I^\omega V=\bigcap_m I^mV,\qquad
 \widehat V=\ilim_mV/I^mV.
\]
The quotients, intersection and completion are computed termwise
in the category of inverse sequences of right $R$-modules.
This is the image-filtration completion of \cite[Section~1.2]{CI}.
The next proposition combines that framework
\cite[Proposition~1.1 and Corollary~2.4]{CI} with
\eqref{eq:GP-top}. We include the module calculation explicitly.

\begin{proposition}\label{prop:two-obstructions}
Assume $d(\A)\leq\nu<\infty$, and put
$V_n=\Ext_\A^\nu(G,X_n)$ and $C=\coker(V\to\widehat V)$.
There is a natural short exact sequence
\begin{equation}\label{eq:two-obstructions}
 0\to Q\ilim^1I^\omega V\to R^{\nu+1}\ilim X
 \to Q\ilim C\to0.
\end{equation}
The filtration, completion and module limits are taken before
applying $Q$.
\end{proposition}
\begin{proof}
Every $m$-step transition in $V/I^mV$ is zero, so its shift
satisfies $\sigma^m=0$. Its difference map has inverse
$1+\sigma+\cdots+\sigma^{m-1}$.
These inverses commute with the quotient maps between filtration
stages. Products commute with ordinary limits, so the difference
map of $\widehat V$ is the inverse limit of these isomorphisms.
It is therefore invertible, and
\[
 \ilim\widehat V=\ilim^1\widehat V=0.
\]
The kernel of $V\to\widehat V$ is $K=I^\omega V$.
Put $T=V/K$. Since $T\hookrightarrow\widehat V$ and inverse limit
is left exact, $\ilim T=0$.
The long exact sequence of $0\to K\to V\to T\to0$, together with
$\ilim^2K=0$ for module inverse sequences, gives
\[
 0\to\ilim^1K\to\ilim^1V\to\ilim^1T\to0.
\]
The sequence $0\to T\to\widehat V\to C\to0$ and the two
vanishings above give $\ilim C\cong\ilim^1T$.
Substitute this isomorphism, apply the exact functor $Q$, and use
\eqref{eq:GP-top} to obtain \eqref{eq:two-obstructions}.
\end{proof}

Thus $R^{\nu+1}\ilim X=0$ if and only if both outer terms of
\eqref{eq:two-obstructions} vanish. Surjectivity of $V\to\widehat V$
removes the right-hand term, but does not control $Q\ilim^1I^\omega V$:
the intersection $I^\omega V$ need not have epimorphic bonding maps.
General top-degree vanishing requires both terms to vanish for the
actual Ext sequences of all epimorphic $X$.


\begin{thebibliography}{99}
\bibitem{CI}M. Carelli and S. O. Ivanov,
\emph{Transfinite version of the Mittag--Leffler condition for the vanishing
of the derived limit}, preprint,
\href{https://arxiv.org/abs/2310.02716v4}{arXiv:2310.02716v4}, 2024.
\bibitem{HX}A. Hogadi and C. Xu,
\emph{Products, homotopy limits and applications},
\href{https://arxiv.org/abs/0902.4016v1}{arXiv:0902.4016v1} (2009).
\bibitem{Jensen}C. U. Jensen,
\emph{Les foncteurs d\'eriv\'es de $\varprojlim$ et leurs applications
en th\'eorie des modules}, Lecture Notes in Mathematics \textbf{254},
Springer-Verlag, Berlin--New York, 1972.
\bibitem{PosCoalg}L. Positselski,
\emph{Homological full-and-faithfulness of comodule inclusion and
contramodule forgetful functors}, Glasgow Math. J. \textbf{68} (2026), 86--119.
\href{https://doi.org/10.1017/S001708952500014X}{doi:10.1017/S001708952500014X}.
\href{https://arxiv.org/abs/2301.09561v4}{arXiv:2301.09561v4}.
\bibitem{PosRoos}L. Positselski,
\emph{Roos axiom holds for quasi-coherent sheaves},
Appl. Categ. Structures \textbf{34} (2026), article 13.
\href{https://doi.org/10.1007/s10485-025-09845-9}{doi:10.1007/s10485-025-09845-9}.
\href{https://arxiv.org/abs/2407.13651v4}{arXiv:2407.13651v4}.
\bibitem{Roos}J.-E. Roos,
\emph{Derived functors of inverse limits revisited},
J. Lond. Math. Soc. (2) \textbf{73} (2006), 65--83.
\bibitem{Salch}A. Salch,
\emph{Derived functors of product and limit in the category of comodules
over the dual Steenrod algebra}, Israel J. Math. \textbf{271} (2026), 741--794.
\href{https://doi.org/10.1007/s11856-025-2777-3}{doi:10.1007/s11856-025-2777-3}.
\href{https://arxiv.org/abs/2303.08796v3}{arXiv:2303.08796v3}.
\bibitem{Stacks}The Stacks Project Authors, \emph{The Stacks Project},
\url{https://stacks.math.columbia.edu}, Tag~077P.
\bibitem{VV}B. Veli\v{c}kovi\'c and A. Vignati,
\emph{Non-vanishing higher derived limits},
Commun. Contemp. Math. \textbf{26} (2024), no.~7, 2350031,
\href{https://doi.org/10.1142/S0219199723500311}{doi:10.1142/S0219199723500311}.
\bibitem{Virili}S. Virili, \emph{On the exactness of products in the localization of (Ab.4*) Grothendieck categories}, J. Algebra \textbf{470} (2017), 37--67.
\end{thebibliography}
\end{document}